\documentclass{amsart}
                        \input xypic

\usepackage{amsthm, amsfonts, amssymb, amsmath, latexsym, enumerate, times}
\usepackage[all]{xy}
\usepackage[latin1]{inputenc}
\usepackage{mathrsfs}
\usepackage{array}
\usepackage{comment}
\usepackage{paralist}
\usepackage{color}
\usepackage{tikz-cd}
\usepackage{xypic}

\newtheorem{theorem}{Theorem}[section]
\newtheorem{lemma}[theorem]{Lemma}

\newtheorem{corollary}[theorem]{Corollary}
\newtheorem{proposition}[theorem]{Proposition}

\theoremstyle{definition}
\newtheorem{definition}[theorem]{Definition}
\newtheorem{remark}[theorem]{Remark}

\newcommand{\Seg}{\rm Seg}

\renewcommand{\L}{{\mathcal L}}

\newcommand{\pgrass}{{\mathbb G}}
\newcommand{\G}{{\mathbb G}}
\newcommand{\KK}{{\mathbb K}}

\newcommand{\pspace}{{\mathbb P}}
\newcommand{\PP}{{\mathbb P}}

\def\geq{\geqslant}
\def\leq{\leqslant}

\begin{document}
 
\title[Unirationality of some varieties]{Unirationality of varieties described by families of projective hypersurfaces}

\author{Ciro Ciliberto}
\address{Dipartimento di Matematica, Universit\`a di Roma Tor Vergata, Via O. Raimondo,
 00173 Roma, Italy}
\email{cilibert@axp.mat.uniroma2.it}

\author{Duccio Sacchi}
\address{Via Filippo Turati 90, 53014 Monteroni d'Arbia, Siena, Italy}
\email{ ducciosacchi89@gmail.com}
 
 
 
\maketitle

\begin{abstract} Let $\mathscr{X}\to W$ be a flat family of generically irreducible hypersurfaces of degree $d\geq 2$ in $\PP^n$ with singular locus of dimension $t$,  with $W$ unirational of dimension $r$. We prove that if $n$ is large enough with respect to $d$, $r$  and $t$, then $\mathscr{X}$ is unirational. This extends results in \cite {Pre, HMP}.   
\end{abstract}

\section*{Introduction} 

A classical theorem by U. Morin says that if $X$ is a general hypersurface of degree $d$ in $\PP^n$ and $n$ is large enough with respect to $d$, then $X$ is unirational (see \cite{Mor1} and also \cite {Cil}). This result has been extended by A. Predonzan in \cite {Pre} to any hypersurface $X$ of degree $d$ in $\PP^n$, even singular in dimension $t$: $X$  turns out to be unirational provided $n$ is large enough with respect to $d$ and $t$ (see Theorem \ref {thm:pred1} for a precise statement). This theorem has been rediscovered by J. Harris, B. Mazur and R. Pandharipande in \cite {HMP}, although they give a lower bound for $n$ that is worse than Predonzan's one. 

The purpose of this paper is to prove an extension of Predonzan's result, namely Theorem 
\ref {thm:pred2}, that asserts that if $\mathscr{X}\to W$ is a flat family of hypersurfaces of degree $d\geq 2$ in $\PP^n$, whose general member is irreducible and singular in dimension $t$, and $W$ is irreducible, unirational of dimension $r$, if $n$ is large enough with respect to $d,r,t$, then $\mathscr{X}$ is unirational. The case $d=2$ is already contained in \cite{conf:unisec}. This theorem, for instance, implies, under suitable numerical conditions, the unirationality of hypersurfaces in Segre products of projective spaces. 

As for the proof,    Predonzan shows in \cite {Pre} that if  $X\subset \PP^n$ is an irreducible hypersurface of degree $d\geq 2$, defined over a field $\KK$ of characteristic zero, containing  a $k$--plane $\Lambda$ along which $X$ is smooth, and if $k$ is large enough with respect to the degree $d$, then $X$ is unirational over the extension of $\KK$ with the Pl\"ucker coordinates of $\Lambda$ (see Theorem \ref {thm:pred} for a precise statement). The key step in our proof is to show that if  $\mathscr{X}\to W$ is a flat family of generically irreducible hypersurfaces of degree $d\geq 2$ in $\PP^n$,  with $W$ irreducible of dimension $r$,  and $n$ is large enough with respect to $d,r,k$, then there is a rationally determined $k$--plane over the generic hypersurface of the family (see Remark \ref {rem:ratdet} for the meaning of being \emph{rationally determined}). This is the so called Section Lemma (see Lemma \ref {rat_sec_F} below). Theorem \ref {thm:pred2} follows by this result and the aforementioned Predonzan's theorem, in view of a unirationality criterion by L. Roth (see Proposition \ref {criterion}). The proof of the Section Lemma   is inspired to a beautiful and elegant idea of F. Conforto's in \cite {conf:unisec}, and it uses a birational description of 
the Fano scheme of $k$--planes  in a projective hypersurface, contained in \S 
\ref {sec:fano}. This in turn requires some preliminaries about Grassmannians contained  in \S \ref {sec:grass}, that essentially appear in a paper by J. G. Semple \cite {sempl_rep_gr}, and which we expose here for the reader's convenience. 

This paper extends some of the results by the second author in his Ph. D. Thesis \cite{Sacchi}. 

In this paper we work over an algebraically closed field $\KK$ of characteristic zero.

\medskip

{\bf Acknowledgements:} Ciro Ciliberto is a member of GNSAGA of INdAM and he acknowledges support from  the MIUR Excellence Department Project awarded to the Department of Mathematics, University of Rome Tor Vergata, 
CUP E83C18000100006.

\section{Some preliminaries on Grassmannians}\label{sec:grass}

In this section we expose some preliminaries on Grassmann varieties, following   \cite{sempl_rep_gr}. 

\subsection{} Let $\pgrass(k,n)$ be the Grassmann variety of $k$--planes in $\pspace^n=\pspace(V)$, where $V$ is a $\KK$--vector space of dimension $n+1$.
One has $\G(k,n)\cong \G(n-k-1,n)$, hence, without loss of generality, we may and will assume $2k<n$.  

The variety $\pgrass(k,n)$ is naturally embedded in $\PP^{N(k,n)}$, with $N(k,n)=\binom{n+1}{k+1}-1$,  via the \emph{Pl\"ucker embedding}. Explicitely, in coordinates, we have the following. Fix a basis $B=\{e_1, \ldots, e_{n+1}\}$ of $V$. Then we can associate to any $k$--plane $\Lambda$ of $\PP^n=\PP(V)$ a $(k+1)\times (n+1)$ matrix 
$$
M_\Lambda=\begin{bmatrix}
v_{1,1}   & \ldots & v_{1,n+1}    \\
\vdots    &        & \vdots       \\
v_{k+1,1} & \ldots & v_{k+1, n+1} \\
\end{bmatrix}
$$
whose rows are the coordinate vectors with respect to the basis $B$ of $k+1$ vectors corresponding to independent points of $\Lambda$. Two matrices $M$ and $M'$ represent the same $k$--plane if and only if there exists $A \in {\rm GL}(k+1,\KK)$ such that $M=AM'$.

The homogeneous coordinates of the point in $\pspace^{N(k,n)}$ corresponding to $\Lambda$ are given by the minors of order $k+1$ of $M_{\Lambda}$. They depend only on $\Lambda$ and not on the matrix $M_{\Lambda}$. These are the \emph{Pl\"ucker coordinates} of $\Lambda$, and we denote them by $z_I$ where $I=(i_1, \ldots, i_{k+1})$ is a multi-index with $1\leq i_1< i_2 < \ldots < i_{k+1} \leq n+1$ denoting the order of the columns of $M_{\Lambda}$ which determine the corresponding minor. The Pl\"ucker coordinates are lexicographically ordered. 

If we consider the subset $U$ of points of $\pspace^{N(k,n)}$ where the first coordinate $z_{1,\ldots, k+1}$ is different from zero, each point $\Lambda \in U \cap \pgrass(k,n)$ represents a $k$--plane such that there is and associated matrix $M_\Lambda$ to $\Lambda$ that can be uniquely written in the form
\begin{equation*}
\begin{bmatrix}
1      & 0 & \ldots & 0      & v_{1,k+2}   & \ldots & v_{1,n+1}    \\
0      & 1 & \ldots & 0      & v_{2,k+2}   & \ldots & v_{2,n+1}    \\
\vdots &   & \ddots & \vdots & \vdots      &        & \vdots       \\
0      & 0 & \ldots & 1      & v_{k+1,k+2} & \ldots & v_{k+1, n+1} \\
\end{bmatrix}.
\end{equation*}
From this description it follows that $U \cap \pgrass(k,n)$ is isomorphic to $\mathbb{A}^{(k+1)(n-k)}$, where the coordinates are the (lexicographically ordered) $v_{i,j}$'s, with $1\leq i\leq k+1$ and $k+2\leq j\leq n+1$.  We will soon give a geometric interpretation of this isomorphism (see Proposition \ref {map_fromGtoP} below).

\begin{remark}\label{ext_algebra_open}
We can describe geometrically the open subset $U\cap \pgrass(k,n)$: it is the set of  $k$--planes of $\pspace^n$  that  do not intersect the $(n-k-1)$--plane spanned by the points corresponding to $e_{k+2}, \ldots, e_{n+1}$. Similarly, for any choice of a totally decomposable element of $\wedge^{n-k}V$ (i.e., a vector which can be expressed as $v_1\wedge \ldots \wedge v_{n-k}$) we can  construct a birational map between $\pgrass(k,n)$ and $\pspace^{(k+1)(n-k)}$.
\end{remark}

\subsection{} Now we set $M(k,n)=(k+1)(n-k)$ and consider $\pspace^{M(k,n)}$ with homogeneous coordinates given by $y$ and $x_{i,j}$ for $i=1, \ldots, k+1$ and $j=k+2, \ldots, n+1$: in this setting the affine space with coordinates $x_{i,j}$ for $i=1, \ldots, k+1$ and $j=k+2, \ldots, n+1$ is the complement of the hyperplane $H$ with equation $y=0$. We define the rational map
\begin{equation}\label{map_fromPtoG}
\psi_{k,n}: \pspace^{M(k,n)} \dashrightarrow \pspace^{N(k,n)}
\end{equation}
sending the point with coordinates $[y, x_{1,k+2}, \ldots, x_{k+1,n+1}]$ to the point whose coordinates are the minors of order $k+1$ of the matrix
\begin{equation}
\begin{bmatrix}\label{matrix_fromPtoG}
y      & 0 & \ldots & 0      & x_{1,k+2}   & \ldots & x_{1,n+1}    \\
0      & y & \ldots & 0      & x_{2,k+2}   & \ldots & x_{2,n+1}    \\
\vdots &   & \ddots & \vdots & \vdots      &        & \vdots       \\
0      & 0 & \ldots & y      & x_{k+1,k+2} & \ldots & x_{k+1, n+1} \\
\end{bmatrix}.
\end{equation}

If we consider the open subset $U' = \{y \neq 0\} \subset \pspace^{M(k,n)}$, $\psi_{k,n}|_{U'}$ is the inverse isomorphism  of the one described above between $U \cap \pgrass(k,n)$ and $\mathbb{A}^{M(k,n)}$. So the image of $\psi_{k,n}$ is $\G(k,n)$. 

We will describe the linear system $\mathfrak{d}_{k,n}$ of hypersurfaces associated to the map $\psi_{k,n}$. It corresponds to the vector space $W \subset H^0(\pspace^M,\mathscr{O}_{\pspace^M}(k+1))$ spanned by $y^{k+1}$ and by the forms
\begin{equation}\label{hyp_sl_expl}
y^{k+1-r}D^r_{i_1,\ldots,i_r;j_1, \ldots, j_r}
\end{equation}
for $r=1, \ldots, k+1$, with $1 \leq i_1 < \cdots <i_r \leq k+1$ and 
$k+2 \leq j_1 < \cdots < j_r \leq n+1$, where  $D^r_{i_1,\ldots,i_r; j_1, \ldots, j_r}$ denotes the minor of order $r$  of the matrix
\begin{equation}\label{eq:matt}
\begin{bmatrix}
x_{1,k+2}   & \ldots & x_{1,n+1}    \\
x_{2,k+2}   & \ldots & x_{2,n+1}    \\
\vdots      &        & \vdots       \\
x_{k+1,k+2} & \ldots & x_{k+1, n+1} \\
\end{bmatrix}
\end{equation}
determined by the rows of place $i_1,\ldots,i_r$ and by the columns of place $j_1, \ldots, j_r$.

For a fixed $r = 1, \ldots, k+1$, we define $m_r$ as the number of the minors of type $D^r$. Thus, $m_r=\binom{k+1}{r}\binom{n-k}{r}$.

Note that the subvariety of $H$ defined by the $2\times 2$ minors of the matrix \eqref {eq:matt} is a Segre variety $\Seg(k,n-k-1)\cong \PP^k\times \PP^{n-k-1}$ (see \cite [p. 98] {har}).

We want to geometrically characterize the  linear system $\mathfrak{d}_{k,n}$. Before doing that, we need the following:

\begin{lemma}\label{lem:fac} Let $r\geq 1$. There is no hypersurface of degree $r+2$ in $\PP^r$ with multiplicity at least $r+1$ at $r+1$ independent assigned points of $\PP^r$.
\end{lemma}

\begin{proof} Consider the linear system of $\PP^r$ of hypersurfaces of degree $r$ with multiplicity at least $r-1$ at the $r+1$ independent assigned points.  This is well known to be a homaloidal linear system, determining a birational map $\omega: \PP^r\dasharrow \PP^r$, such that the counterimages of the lines of the target $\PP^r$  are the rational normal curves in the domain $\PP^r$ passing through the $r+1$ independent assigned points.  The intersection of a  hypersurface of degree $r+2$ in $\PP^r$ with multiplicity at least $r+1$ at the $r+1$ independent assigned points with these rational normal curve off the $r+1$ independent assigned points is $-1$. So such a hypersurface is empty and the assertion follows.
\end{proof}

Next we can give the desired geometric description of the  linear system $\mathfrak{d}_{k,n}$:

\begin{proposition}\label{lin_sist_hyp}
The linear system $\mathfrak{d}_{k,n}$ consists of the hypersurfaces of degree $k+1$ in $\pspace^{M(k,n)}$ passing with multiplicity at least $k$ through the Segre variety $\Seg(k,n-k-1)$ contained in the hyperplane $H$ with equation $y=0$ and defined by the $2	\times 2$ minors of the matrix \eqref {eq:matt}. 
\end{proposition}

\begin{proof}
The linear system $\mathfrak{d}_{k,n}$ has a base locus scheme $B_1$. By looking at  the basis of $W$ in \eqref{hyp_sl_expl},
$B_1$ is defined by the equations
$$
y=0, \quad D^{k+1}_{1,\ldots,{k+1}; j_1, \ldots, j_{k+1}}, \quad \text{for all}\quad k+2 \leq j_1 < \ldots < j_r \leq n+1.
$$
Then $B_1$ is the $(k-1)$--th secant variety of $\Seg(k,n-k-1)$ defined by the $2\times 2$ minors of the matrix \eqref {eq:matt} inside $H$ (see \cite [p. 99]{har}). Moreover 
$\mathfrak{d}_{k,n}$ is  the whole linear system of hypersurfaces of degree $k+1$ of 
$\pspace^{M(k,n)}$ containing $B_1$.

For each $r=2, \ldots, k$, we can also consider the subscheme $B_r$ of $B_1$ consisting of those points of $B_1$ where all hypersurfaces of $\mathfrak{d}_{k,n}$ have multiplicity at least $r$. Looking again at  the basis of $W$ in \eqref{hyp_sl_expl}, it is immediate that $B_r$ is the $(k-r)$--secant variety of $\Seg(k,n-k-1)$, for $r=2, \ldots, k$ (see again \cite [p. 99]{har}). Note that $B_1$ itself has points of multiplicity at least $r$ along $B_r$, for all $r=2, \ldots, k$. In particular, each hypersurface in the linear system $\mathfrak{d}_{k,n}$ passes with multiplicity $k$ through the Segre variety $\Seg(k,n-k-1)$ in $H$, which is $B_k$.

Conversely, let $F$ be a hypersurface of degree $k+1$ in $\pspace^{M(k,n)}$ passing with multiplicity at least $k$ through the Segre variety $\Seg(k,n-k-1)$ contained in the hyperplane $H$ and defined by the $2\times 2$ minors of the matrix \eqref {eq:matt}.
Then we claim that $F$ contains the $(k-1)$--th secant variety of $\Seg(k,n-k-1)$, hence $F$ belongs to $\mathfrak{d}_{k,n}$. Indeed, this is clear for $k=1$, so we may assume $k\geq 2$, 
in which case the claim follows right away by  Lemma \ref {lem:fac}. 
\end{proof}

\subsection{} Next we need a description of the \emph{osculating spaces} to the Grassmann varieties (for the concept of osculating spaces see \cite[p.141]{rus}).  First we need some lemmata.

\begin{lemma}\label{rnc_in_Gr}
Let $\Seg(1,k)$ be a Segre variety in $\pspace^n$, with $n\geq 2k+1$. Let $\phi: \mathbb{P}^1 \rightarrow \pgrass(k,n)$ be the morphism which sends a point $p$ to the $k$--plane $\{p\}\times \pspace^{k}\subset \Seg(1,k)$ in $\PP^n$. Then the image of $\phi$ is a rational normal curve of degree $k+1$ inside $\pgrass(k,n)$.
\end{lemma}

\begin{proof}
We can assume $n=2k+1$.  If $[x_0, x_1]$ are  homogeneous coordinates of $\pspace^1$ and $z_{ij}$, $i=0,1$ and $j=0,\ldots, k$, are the homogenous coordinates of $\pspace^{2k+1}$, we can assume that $\phi$ is the map which sends $[\alpha_0,\alpha_1]$ to the $k$-plane whose equations in $\pspace^{2k+1}$ are $\alpha_1z_{0j}-\alpha_0z_{1j}=0$ for $j=0,\ldots,k$. In particular, the image of a point $[x_0,x_1]$ under $\phi$ is the point of $\pgrass(k,2k+1)$ whose coordinates are the minors of maximal order of the matrix
\begin{equation*}\label{matrix_rnc}
\begin{bmatrix}
x_0     & 0     & \ldots & 0     & x_1   & 0     & \ldots & 0  \\
0         & x_0 & \ldots & 0     & 0       & x_1 & \ldots & 0  \\
\vdots  &        & \ddots&       &          &        & \ddots & \vdots\\
0         & 0     & \ldots & x_0 & 0       & 0     & \ldots & x_1  \\
\end{bmatrix}.
\end{equation*}
There are only $k+2$ non-vanishing Pl\"ucker coordinates of this $k$--plane and they have as entries the monomials of degree $k+1$ in $x_0$ and $x_1$. The assertion follows.
\end{proof}

We can generalize the above result:

\begin{lemma}\label{gen_rnc_in_GR}
Let $k,r,n$ be positive integers with $k>r$. Let $\Seg(1,r)$ be a Segre variety in $\pspace^{n}$, with $n\geq k+r+1$,  and let $\Pi$ be a $(k-r-1)$--plane,  which does not intersect the $(2r+1)$--plane spanned by $\Seg(1,r)$. Let $\phi: \mathbb{P}^1 \rightarrow \pgrass(k,n)$ be the morphism which sends a point $p$ to the $k$--plane spanned by $\Pi$ and by the $r$--plane in $\Seg(1,r)$ given by $\{p\} \times \pspace^r$. Then its image is a rational normal curve of degree $r+1$ inside $\pgrass(k,n)$.
\end{lemma}

\begin{proof}
We can assume $n=k+r+1$. Let $\Seg(1,r)$ be a Segre variety inside the $(2r+1)$--plane $L$ given by the vanishing of the last $k-r$ homogeneous coordinates of $\pspace^{k+r+1}$, and let $\Pi$ be the $k-r-1$ plane given by the vanishing of the first $2r+2$ coordinates. Then we can associate to the point $[x_0, x_1]$ of $\pspace^1$ the matrix whose rows span the join of $\Pi$ and $[x_0,x_1]\times \pspace^r$. This is the $(k+1)\times(k+r+2)$ matrix
\[
\left[ 
\begin{array} { c|c } 
A_{x_0,x_1} & 0_1 \\ 
\hline 
0_2   & I_{k-r}\\
\end{array} 
\right],
\]
where $A_{x_0,x_1}$ is the $(r+1)\times(2r+2)$ matrix associated to the $r$-plane $[x_0,x_1]\times\pspace^r$ in $L$, $0_1$ is the $(r+1)\times(k-r)$ zero matrix, $0_2$ is the $(k-r)\times(2r+2)$ zero matrix and $I_{k-r}$ is the $(k-r)$ identity matrix. As in the proof of Lemma \ref {rnc_in_Gr}, we see that  the only non-vanishing Pl\"ucker coordinates of the point of the Grassmaniann associated to this matrix are given by the monomials of order $r+1$ in $x_0$ and $x_1$. The assertion follows.
\end{proof}

\begin{lemma}\label{lem:ffac} Let $L_1,L_2$ be two distinct $k$--planes in $\mathbb{P}^n$ intersecting in a $(k-r)$--plane $M$, with $1\leq r\leq k+1$. Then there is a Segre variety $\Seg(1,r-1)$ in $\PP^n$ such that there are two distinct point $p_1,p_2\in \PP^1$ such that 
$L_i\cap \Seg(1,r-1)=\{p_i\}\times \PP^{r-1}$, for $i=1,2$.
\end{lemma}

\begin{proof} 
Projecting from $M$ to a $\pspace^{n-k+r-1}$, the images of $L_1,L_2$ are two disjoint $(r-1)$--planes $L'_1$ and $L'_2$. If we fix an isomorphism $\tau: L'_1 \rightarrow L'_2$, the variety defined as the union of the lines joining $p \in L'_1$ to $\tau(p) \in L'_2$ is the desired Segre variety.
\end{proof}

The following proposition describes the osculating spaces of Grassmannians.

\begin{proposition}\label{osc_pl_G}
Let $\Lambda_0$ be a point of $\pgrass(k,n)$ and let $1\leq r \leq k$. Then the $r$--osculating space $T^{(r)}_{\G(k,n),\Lambda_0}$ to $\G(k,n)$ at $\Lambda_0$  is the linear space  spanned by the Schubert variety 
\[
W_{r,\Lambda_0}=\left\{\Lambda \in \pgrass(k,n) | \dim (\Lambda \cap \Lambda_0) \geq k-r \right\}
\]
and one has
\begin{equation}\label{eq:dim}
\dim(T^{(r)}_{\G(k,n),\Lambda_0})=\sum_{i=1}^r\binom{k+1}{i}\binom{n-k}{i}.
\end{equation}
\end{proposition}

\begin{proof} First of all we claim that $W_{r,\Lambda_0}\subseteq T^{(r)}_{\G(k,n),\Lambda_0}$. 
To prove this note that by Lemmata \ref{gen_rnc_in_GR} and \ref {lem:ffac}, for any $k$--plane $\Lambda$  intersecting $\Lambda_0$ in a linear space of dimension at least $k-r$, we can construct a rational normal curve of degree $r$  in $\G(k,n)$ passing through $\Lambda_0$ and $\Lambda$. Such a curve must be contained in $T^{(r)}_{\G(k,n),\Lambda_0}$ and this proves the claim. 

To prove that $T^{(r)}_{\G(k,n),\Lambda_0}=\langle  W_{r,\Lambda_0} \rangle$, we will compute the dimensions of both $T^{(r)}_{\G(k,n),\Lambda_0}$ and $\langle  W_{r,\Lambda_0} \rangle$ and we will prove they are equal. 

First let us prove \eqref {eq:dim}. 
Without loss of generality, we may assume that $\Lambda_0$ is spanned by the points corresponding to the vectors $e_1,\ldots, e_{k+1}$ of the basis $B$ of $V$, so that $\Lambda_0$ is the point where only the first Pl\"ucker coordinate is different from zero. Consider the local parametrization of $\G(k,n)$ around $\Lambda_0$ given by the restriction of the map $\psi_{k,n}$  as in \eqref{map_fromPtoG} to $\mathbb{A}^{M(k,n)}=\PP^{M(k,n)}\setminus H$, so that $\psi_{k,n}$ maps the origin of $\mathbb{A}^{M(k,n)}$ to $\Lambda_0$. The $r$--osculating space $T^{(r)}_{\G(k,n),\Lambda_0}$ is spanned by the points that are derivatives up to order $r$ of the  parametrization at the origin. 

Each coordinate function of $\psi_{k,n}$ is given by a minor $D^s$ as above (in the affine coordinates $x_{i,j}$, for $i=1, \ldots, k+1$ and $j=k+2, \ldots, n+1$, of $\mathbb{A}^{M(k,n)}$). The derivatives up to order $r$ of the minors $D^s$ with $s \geq r+1$ vanish at $0 \in \mathbb{A}^{M(k,n)}$. Hence  $T^{(r)}_{\G(k,n),\Lambda_0}$ has dimension at most $\sum_{i=1}^rm_i$, where we recall that $m_i=\binom{k+1}{i}\binom{n-k}{i}$ is the number of  the $D^i$'s.

Moreover, for each minor $D^s$ with $s \leq r$, there exists a derivative of order $s$ of the  parametrization at the origin such that all of its coordinates, except the one corresponding to $D^s$, vanish. This implies \eqref {eq:dim}.

Next we compute the dimension of $\langle  W_{r,\Lambda_0} \rangle$ and prove that it equals the right hand side of  \eqref {eq:dim}. Let $\Lambda$ be an element of $W_{r,\Lambda_0}$. It is spanned by $k+1$ points, and we may assume the first $k-r+1$ of them lie on $\Lambda_0$.  Then the Pl\"ucker coordinates of $\Lambda$ are given by the maximal minors of a matrix $M_{\Lambda}=[v_{i,j}]_{i=1, \ldots, k+1; j=1, \ldots, n+1}$ where $v_{i,j}=0$ if $i \in \{1, \ldots, k-r+1\}$ and $j\in\{k+2 ,\ldots, n+1\}$. Moreover, varying $\Lambda$ in $W_{r,\Lambda_0}$ we may consider the non--zero $v_{i,j}$ as variables. 

The vanishing maximal minors of a matrix of type $M_{\Lambda}$ are those  involving at most $r+1$ of the last $n-k$ columns. Hence  their number is
$$
c=\sum_{i=r+1}^{k+1} {{n-k}\choose i}{{k+1}\choose i}=\sum_{i=r+1}^{k+1} m_i
$$
and therefore
$$
\dim (\langle  W_{r,\Lambda_0} \rangle)=N(k,n)-c.
$$
On the other hand we have
$$
N(k,n)=\sum_{i=1}^{k+1}m_i
$$
hence
$$
\dim (\langle  W_{r,\Lambda_0} \rangle)=\sum_{i=1}^{r}m_i=\dim (T^{(r)}_{\G(k,n),\Lambda_0}),
$$
as desired.  \end{proof}

\subsection{} Next we  give the announced  geometric description of the isomorphism of  $U \cap \pgrass(k,n)$ with $\mathbb{A}^{M(k,n)}$.

\begin{proposition}\label{map_fromGtoP}
Let $\Pi$ be an element of $\pgrass(n-k-1,n)$ and $W_\Pi$ the Schubert variety 
\begin{equation*}\label{vertex_proj_fromGtoP}
\left\{\Lambda \in \pgrass(k,n) | \dim (\Lambda \cap \Pi) \geq 1\right\}.
\end{equation*}
Then the projection $\varphi: \pgrass(k,n) \dashrightarrow \pspace^{M(k,n)}$ from the linear space spanned by $W_\Pi$ is the inverse map of a $\psi_{k,n}:\pspace^{M(k,n)} \dashrightarrow \pgrass(k,n)$ as in \eqref{map_fromPtoG}.
\end{proposition}

\begin{proof}
We use the notation of Lemma $\ref{lin_sist_hyp}$. First of all we observe that the linear system $\mathfrak{d}_{k,n}$ contains the linear system of hyperplanes of $\pspace^{M(k,n)}$ as a subsystem: this is $kH+|\pi|$, where $\pi$ is any hyperplane. Via the map $	\psi_{k,n}$ the hypersurfaces of $\mathfrak{d}_{k,n}$ are sent to hyperplane sections of $\pgrass(k,n)$. Thus  the inverse of $\psi_{k,n}$ is a projection whose centre is the intersection of all hyperplanes of $\pspace^{N(k,n)}$ whose intersection with $\G(k,n)$ contains $\psi_{k,n}(H)$ with multiplicity at least $k$.

The image of $H$ under $\psi_{k,n}$ is the Grassmannian $\pgrass_0=\pgrass(k, n-k-1)$ of all subspaces of dimension $k$ contained in a fixed subspace $\Pi$ of $\PP^n$ dimension  $n-k-1$. Indeed, if we set $y=0$ in \eqref{matrix_fromPtoG}, we obtain the Pl\"ucker embedding associated to a $(k+1)\times(n-k)$ matrix.

A hyperplane $H'$ in $\pspace^{N(k,n)}$ contains $\pgrass_0$ with multiplicity at least $k$ if and only if $H'$ contains $T^{(k-1)}_{\G(k,n),P}$ for any $P \in \pgrass_0$ and the centre of the projection is the intersection of these hyperplanes.  
Then from Proposition \ref{osc_pl_G}  the centre of  projection is the linear span of $W_{\Pi}$. This proves the assertion. \end{proof}

\begin{remark} With a dimension count similar to the one at the end of Proposition \ref {osc_pl_G}, one checks that the linear space spanned by $W_\Pi$ has dimension $N(k,n)-m_1-1=N(k,n)-(k+1)(n-k)-1=N(k,n)-M(k,n)-1$. This fits with the result of Proposition \ref {map_fromGtoP}.
\end{remark}

\begin{remark} From the above considerations it follows that the birational map $\psi_{n,k}$ induces an isomorphism between $\pspace^{M(k,n)}$ minus a hyperplane $H$ and $\pgrass(k,n)$ minus a hyperplane section $\mathfrak H'$, precisely the hyperplane section corresponding to the hypersurface $(k+1)H$ in $\mathfrak d_{k,n}$. Looking at the proof of Proposition \ref {map_fromGtoP}, we see that $\mathfrak H'$ contains $\pgrass_0$ with multiplicity  $k+1$, hence  it contains $T^{(k)}_{\G(k,n),P}$ for any $P \in \pgrass_0$. From Proposition \ref{osc_pl_G} one deduces that $\mathfrak H'$ coincides with the set of all $\Lambda\in \G(k,n)$ that have non--empty intersection with the $(n-k-1)$--plane $\Pi$.
We will call the hyperplane $H'$ cutting out such a $\mathfrak H'$ on $\G(k,n)$ a \emph{$k$--osculating hyperplane} to $\G(k,n)$. 
\end{remark}

\begin{lemma}\label{osc_hyperplanes_grass}
Let $\check{\pspace}^{N(k,n)}$ be the dual space of $\pspace^{N(k,n)}$. Then the $k$--osculating hyperplanes to $\pgrass(k,n)$ are parametrized by a $\pgrass(n-k-1,n)$ in $\check{\pspace}^{N(k,n)}$. In particular, since $\pgrass(n-k-1,n)$ is non--degenerate in $\check{\pspace}^{N(k,n)}$, there is no point of ${\pspace}^{N(k,n)}$  contained in all   $k$--osculating hyperplanes.
\end{lemma}

\begin{proof} We have $\check{\pspace}^{N(k,n)}=\pspace(\wedge^{k+1}\check V)=\pspace(\wedge^{n-k}V)$. 

Let $\Pi$ be a $(n-k-1)$--plane spanned by $n-k$ points corresponding to the vectors $v_1, \ldots, v_{n-k}$ of $V$. A $k$--plane $\Lambda$, spanned by $k+1$ points  corresponding  to the vectors $w_1, \ldots, w_{k+1}$ of $V$, intersects $\Pi$ if and only if the square matrix of order $n+1$ whose rows are $v_1, \ldots, v_{n-k}, w_1, \ldots, w_{k+1}$ has zero determinant. The set of these $k$--planes is the section of $\pgrass(k,n)$ with the $k$--osculating hyperplane of $\pspace^{N(k,n)}$ of equation
\[
\sum_{1\leq i_1 < \cdots < i_{n-k} \leq n+1}S_{i_1,\ldots, i_{n-k}}p_{i_1,\ldots, i_{n-k}}x_{\overline{i_1, \ldots, i_{n-k}}}=0,
\]
where the $p_{i_1,\ldots,i_{n-k}}$'s are the Pl\"ucker coordinates of $\Pi$ in $\pgrass(n-k-1,n)$, the $x_{\overline{i_1, \ldots, i_{n-k}}}$'s are the homogeneous coordinates of $\pspace^{N(k,n)}$, where we denote by ${\overline{i_1, \ldots, i_{n-k}}}$ the $(k+1)$--tuple of indices obtained by deleting $\{i_1, \ldots, i_{n-k}\}$ from $(1, \ldots, n+1)$, and $S_{i_1,\ldots, i_{n-k}}$ is the sign of the permutation $(i_1,\ldots, i_{n-k}, \overline{i_1, \ldots, i_{n-k}})$. 

So the coordinates of this hyperplane in $\check{\pspace}^{N(k,n)}$ are $[S_{i_1,\ldots, i_{n-k}}p_{i_1,\ldots, i_{n-k}}]_{(i_1,\ldots, i_{n-k})}$. The assertion follows. 
\end{proof}

\begin{corollary}\label{general_proj}
Let $X$ be an irreducible subvariety of $\pgrass(k,n)$. Then for a general projection 
$\varphi: \pgrass(k,n) \dashrightarrow \pspace^{M(k,n)}$ as in Proposition \ref {map_fromGtoP}, $X$ is not contained in the indeterminacy locus of $\varphi$ and the restriction of $\varphi$ to $X$ is a birational map of $X$ to its image. 
\end{corollary}

\begin{proof}
Given a projection $\varphi: \pgrass(k,n) \dashrightarrow \pspace^{M(k,n)}$ as in Proposition \ref {map_fromGtoP}, its   indeterminacy locus and the subvariety contracted by the projection are contained in a $k$--osculating hyperplane section of the Grassmannian. By Lemma \ref {osc_hyperplanes_grass}, these hyperplanes vary in a Grassmannian $\pgrass(n-k-1,n)$ in $\check{\pspace}^{N(k,n)}$, and there is no point of ${\pspace}^{N(k,n)}$  contained in all these  hyperplanes. Hence, given the subvariety $X$ in $\G(k,n)$, there is certainly a $k$--osculating hyperplane non containing it. The corresponding projection enjoys the required property. \end{proof}

\section{Fano schemes}\label{sec:fano}

Let $X\subset \PP^n$ be an irreducible projective variety. Given any positive integer $k$, we will denote by $F_k(X)$ the Hilbert scheme of $k$--planes of $\PP^n$ contained in $X$. This is also called the 
\emph{$k$--Fano scheme} of $X$. We will not be interested in the scheme structure on $F_k(X)$, but rather on its support. In particular we will be interested in $F_k(X)$ when $X$ is an irreducible  hypersurface of degree $d\geq 2$ in $\PP^n$.  

 This short section is devoted to prove the following:

\begin{proposition}\label{fano-bi-rational}
Let $X$ be an irreducible hypersurface of degree $d\geq 2$ in $\pspace^n$. Let $\varphi: \pgrass(k,n) \dashrightarrow \pspace^{M(k,n)}$ be a general projection map as in Proposition $\ref{map_fromGtoP}$. Then $\varphi|_{F_k(X)}$ is a birational map on each component of $F_k(X)$ and $\varphi(F_k(X))$ is defined by the vanishing of $\binom{d+k}{k}$ polynomials of degree $d$.
\end{proposition}

\begin{proof}
The first assertion follows directly from Corollary $\ref{general_proj}$.

Let us fix homogeneous coordinates $[x_0,\ldots, x_n]$ in $\PP^n$ and let $f=0$ be the equation of $X$  in this system, with
\[
f(x_0, \ldots,x_n)=\sum_{d_0+\ldots+d_n=d}\alpha_{d_0\ldots d_n}x_0^{d_0}\ldots x_n^{d_n}.
\]

We assume, without loss of generality, that the projection is an isomorphism on the open set $U$ of the Grassmannian where the first Pl\"ucker coordinate is different from zero. For every $\Lambda \in U$ we can give a parametrization $\phi_\Lambda: \pspace^k \rightarrow \Lambda \subseteq \pspace^n$ of $\Lambda$ as

\begin{gather*}
[s_0, \ldots, s_k] \mapsto [s_0, \ldots, s_k] \begin{bmatrix}
1      & 0 & 0 & \ldots & 0 & a_{1,k+2} & a_{1,k+3} & \ldots & a_{1,n+1} \\
0      & 1 & 0 & \ldots & 0 & a_{2,k+2} & a_{2,k+3} & \ldots & a_{2,n+1} \\
\vdots &   &   &        &   &           &           &        &         \\
0      & 0 & 0 & \ldots & 1 & a_{k+1,k+2} & a_{k+1,k+3} & \ldots & a_{k+1,n+1} \\
\end{bmatrix}
\end{gather*}
with $a_{i,j}$, for $1\leq i\leq k+1$, $k+2\leq j\leq n+1$, depending on $\Lambda$. 

Then $f(\phi_\Lambda([s_0, \ldots, s_k]))$ is a form of degree $d$ in $s_0, \ldots, s_k$ with coefficient polynomials in the  $a_{i,j}$'s and in the $\alpha_{d_0 \ldots d_n}$'s. Imposing that $\Lambda$ sits in $F_k(X)$ is equivalent to impose that $f(\phi_\Lambda([s_0, \ldots, s_k]))$ is identically zero as a form in $s_0, \ldots, s_k$. This translates in imposing that the $\binom{d+k}{k}$ coefficients of $f(\phi_\Lambda([s_0, \ldots, s_k]))$ all vanish, and these are linear in the $\alpha_{d_0,\ldots,d_n}$'s and of degree $d$ in the $a_{i,j}$'s. The assertion follows. \end{proof}

\section{Families of hypersurfaces and the section lemma}\label{sec:fam}

In this section we introduce the definition of a \emph{family of hypersurfaces} and we prove a crucial result, the \emph{Section Lemma} \ref {rat_sec_F}, in whose proof we use an idea of Conforto \cite{conf:unisec}, which extends previous work by Comessatti \cite{comes:unisec}.

\subsection{} We start with some definitions. We will denote by $\L_{n,d}$ the linear system of all hypersurfaces of degree $d$ in $\PP^n$, and by $p: \mathcal H_{n,d}\longrightarrow \L_{n,d}$ the universal family, so that $ \mathcal H_{n,d}\subset \L_{n,d}\times \PP^n$ and $p$ is the projection to the first factor.

\begin{definition}\label{family_hypers}
Let $W$ be an irreducible variety. We call a \emph{family of hypersurfaces (of degree $d$ and dimension $n-1$) parametrized by $W$} any morphism $f: \mathscr{X} \rightarrow W$, such that there exists a morphism $g:W \rightarrow \L_{n,d}$ so that the following diagram
\[ \xymatrix{
\mathscr{X}  \ar[r] \ar[d]_{f} &  \mathcal H_{n,d} \ar[d]^{p} \\
 W \ar[r]^{g} & \L_{n,d}
}
\]
is cartesian. In particular, $f: \mathscr{X} \rightarrow W$ is flat. For any point $w\in W$ we will denote by $X_w\subset \PP^n$ the corresponding hypersurface, i.e., the fibre of $f: \mathscr{X} \rightarrow W$ over $w$. 
\end{definition}

\begin{definition}\label{birational_families}
Given two families of hypersurfaces $\mathscr{X}\rightarrow W$ and $\mathscr{Y}\rightarrow T$ as in Definition \ref {family_hypers}, we say that $\mathscr{X}$ is \emph{birationally equivalent} to $\mathscr{Y}$ if there exist two birational maps $f:\mathscr{X} \dashrightarrow \mathscr{Y}$ and $g:W\dashrightarrow T$ such that the diagram
\[
\begin{tikzcd}
\mathscr{X}\arrow[r,dashed, "f"] \arrow[d] & \mathscr{Y}\arrow[d]
\\
W\arrow[r,dashed, "g"] & T
\end{tikzcd}
\]
commutes.
\end{definition}

We will be interested in family of hypersurfaces up to birational equivalence. The following lemma gives us a sort of canonical way of representing a family of hypersurfaces up to birational equivalence.

\begin{lemma}\label{description_fam}
Let $\mathscr{X} \rightarrow W$ be a family of hypersurfaces of degree $d$ in $\PP^n$ with $\dim(W)=r$. Then there is a birationally equivalent family $\mathscr{X}' \rightarrow W'$ such that $W'$ is a dense open subset of a hypersurface in $\pspace^{r+1}$ 
which is birational to $W$ and $\mathscr{X}'\subset W'\times \PP^n$ has equation of the form
\begin{equation}\label{sist_eq_rat_fam}
\sum_{i_1, \ldots, i_d \in \{0, \ldots, n\}}a_{i_1\ldots i_d}(u_0, \ldots, u_{r+1})\prod_{j=1}^dx_{i_j}=0
\end{equation}
where $a_{i_1\ldots i_d}\in H^0(W', \mathcal O_{W'}(\mu))$ for some $\mu \in \mathbb{N}$, for all $i_1, \ldots, i_d \in \{0, \ldots, n\}$.
\end{lemma}

\begin{proof} To give the family $\mathscr{X} \rightarrow W$ is equivalent to give the corresponding morphism $g: W\to \L_{n,d}$. Let $\mathfrak W\subset \PP^{r+1}$ be a hypersurface with a birational map $h: \mathfrak W\dasharrow W$. Then $g'=g\circ h: \mathfrak W\dasharrow \L_{n,d}$ is a rational map, and there is a dense open subset $W'$ of $\mathfrak W$ where $g'$ is defined. Then we have a morphism $g': W'\to \L_{n,d}$ and accordingly we have a family $\mathscr{X}' \rightarrow W'$ that is birationally equivalent to $\mathscr{X} \rightarrow W$. On the other hand, giving $g': W'\to \L_{n,d}$ is equivalent to give a suitable $\binom{n+d}{n}$--tuple of elements $a_{i_1 \ldots i_d} \in H^0(W', \mathscr{O}_{W'}(\mu))$, for all $i_1, \ldots, i_d \in \{0, \ldots, n\}$ and some positive integer $\mu$, so that $\mathscr{X}'\subset W'\times \PP^n$ has equation \eqref {sist_eq_rat_fam}. 
\end{proof}

\subsection{} Next we want to prove the announced \emph{Section Lemma}. 

Let $\mathscr{X}\rightarrow W$ be a family of hypersurfaces of degree $d\geq 2$ in $\PP^n$. We denote by $F_k(\mathscr{X}) \rightarrow W$ the \emph{relative Fano scheme} of $k$--planes in $\PP^n$ contained in fibres of $\mathscr{X}\rightarrow W$. For any point $w\in W$, the fibre of $F_k(\mathscr{X}) \rightarrow W$ over $w$ is $F_k(X_w)$. 

We recall the following result (see \cite{Mor}):

\begin {theorem}\label{thm:mor} Let $k,n,d$ positive integers with $d\geq 2$ and
\begin{equation}\label{eq:mor}
 n \geq    \begin{cases}
      2k+1, & \text{if $d=2$ and $k\geq 2$}\\
 \frac {1}{k+1} {{k+d}\choose d}, & \text{otherwise}. 
    \end{cases}
\end{equation}
Then all hypersurfaces of degree $d$ in $\PP^n$ contain a $k$--plane.
\end{theorem}

Next we consider $\mathscr{X} \rightarrow W$ a family of hypersurfaces of degree $d\geq 2$ in $\mathbb{P}^n$, with  $\dim(W)=r$. We will assume that
\begin{equation}\label{condizione-k-piano}
n>k+\frac{1}{k+1}\left[\binom{d+k}{k}d^r-1\right].
\end{equation}
Then clearly \eqref {eq:mor} holds, hence, by Theorem \ref {thm:mor}, the morphism $F_k(\mathscr{X}) \rightarrow W$ is surjective. By generic flatness, there is a dense open subset  of $W$ over which $F_k(\mathscr{X}) \rightarrow W$  is flat.

We are ready  to prove the \emph{Section Lemma}:

\begin{lemma}[The Section Lemma] \label{rat_sec_F}
Let $\mathscr{X} \rightarrow W$ be a family of hypersurfaces of degree $d\geq 2$ in $\mathbb{P}^n$, with  $\dim(W)=r$ so that \eqref {condizione-k-piano} holds. Then there is a dense open subset $U$ of $W$ such that over $U$ there is a section of  $F_k(\mathscr{X}) \rightarrow W$. \end{lemma} 

\begin{proof} Since the problem is birational in nature, 
by Lemma $\ref{description_fam}$ we may assume that $W$ is a dense open subset  of a hypersurface of degree $\overline{m}$ in $\mathbb{P}^{r+1}$ with equation
$$\phi(u_0, \ldots, u_{r+1})=0.$$

The domain $F_k(\mathscr{X})$ of the Fano family $F_k(\mathscr{X}) \rightarrow W$, that up to shrinking $W$ we may assume to be flat,  is contained in $W \times \mathbb{G}(k,n)$. Consider a general birational projection $\varphi: \pgrass(k,n) \dashrightarrow \pspace^{M(k,n)}$ as in Proposition \ref 
{map_fromGtoP}, that determines a birational map 
$$
\Phi: W\times  \mathbb{G}(k,n) \dasharrow W\times \PP^{M(k,n)}.
$$
By applying Corollary \ref {general_proj} and up to shrinking $W$, we may suppose that  for all $w\in W$, the restriction of $\Phi$ to any irreducible component of  $\{w\}\times F_k(X_w)$ is birational onto its image so that $\Phi$ restricts to a birational map of  $F_k(\mathscr{X})$ to its image, that we denote by $\mathbb{F}_k(\mathscr{X})$, 
contained in $W\times \PP^{M(k,n)}$. By Proposition \ref {fano-bi-rational} we may assume that $\mathbb{F}_k(\mathscr{X})$ is defined by the vanishing of $\binom{d+k}{k}$ equations in $W \times \pspace^{(M(k,n)}$ of the form
\begin{equation}\label{EQ_FAM_FANO}
\sum_{i_1, \ldots, i_d \in \{0, \ldots, n\}}b^\ell_{i_1\ldots i_d}(u_0, \ldots, u_{r+1})\prod_{j=1}^dy_{i_j}=0
\end{equation}
for $\ell=1, \ldots, \binom{d+k}{k}$, and $b^\ell_{i_1\ldots i_d}(u_0, \ldots, u_{r+1}) \in H^0(W, \mathscr{O}_{W}(\mu))$ for a suitable positive integer $\mu$, where the $y_{i}$'s denote  the homogeneous coordinates of $\pspace^{M(k,n)}$.

To prove the lemma we have to prove the existence of a rational section of $F_k(\mathscr{X}) \rightarrow W$ and it clearly suffices to find a rational section $p: W \dashrightarrow \mathbb{F}_k(\mathscr{X})$ of $\mathbb{F}_k(\mathscr{X})\rightarrow W$. Such a rational section is determined by  a suitable $(M(k,n)+1)$-tuple of rational functions on $W$. We may assume that each such rational function is expressed by a homogeneous polynomial in the variables $u_0, \ldots, u_{r+1}$ of a fixed degree $m$ modulo $\phi(u_0, \ldots, u_{r+1})$.

Supposing $m > \overline{m}$, we can choose $M$ independent elements in $H^0(W, \mathscr{O}_W(m))$, where
\begin{equation*}\label{ennegr}
M=\binom{m+r+1}{r+1}-\binom{m-\overline{m}+r+1}{r+1}.
\end{equation*}
These can be identified with $M$ forms $\Psi_1, \ldots, \Psi_M$ of degree $m$, modulo $\phi(u_0, \ldots, u_{r+1})$.

We want to construct a section $p$ by writing its homogeneous coordinates as linear combinations of the $\Psi$'s as above, i.e., by writing them as
\begin{equation*}\label{coord_pt}
p_i=\sum_{j=1}^M \lambda_{i,j} \Psi_j \text{ for $i=0, \ldots, M(k,n)$}
\end{equation*}
where we take the $\lambda_{i,j}$'s as indeterminates. The number of the $\lambda$'s is 
\begin{equation*}\label{num_lambda}
[(k+1)(n-k)+1]M=[(k+1)(n-k)+1]\left[\binom{m+r+1}{r+1}-\binom{m-\overline{m}+r+1}{r+1}\right].
\end{equation*}

We need to find the values of these $\lambda$'s so that $p$ is a section. For this, we have to replace the $y_i$'s in each of the equations \eqref{EQ_FAM_FANO} with the $p_i(u_0, \ldots, u_{r+1})$'s and we have impose that the results identically vanish on $W$, i.e.,  they must be forms in $\KK[u_0, \ldots, u_{r+1}]$ that are divisible by $\phi(u_0, \ldots, u_{r+1})$.

We make the substitution and for each $\ell=1, \ldots, \binom{d+k}{k}$ we have expressions of the sort 
\begin{multline}
\sum_{i_1, \ldots, i_d \in \{0, \ldots, n\}}b^\ell_{i_1\ldots i_d}(u_0, \ldots, u_{r+1})\prod_{j=1}^dp_{i_j}=\\
=\sum_{l_1+\ldots+l_{r+1}=dm+\mu}F^\ell_{l_0\cdots l_{r+1}}(\lambda_{i,j})u_0^{l_0}\cdots u_{r+1}^{l_{r+1}}
\end{multline}
where the homogeneous polynomials that we have after the substitution are of degree $dm+\mu$ with respect to $u_0, \ldots, u_{r+1}$ and the coefficients $F^\ell_{l_0\cdots l_{r+1}}$ are polynomials in the $\lambda$'s. 

Thus, for all $\ell=1, \ldots, \binom{d+k}{k}$, we have to impose that 
\begin{multline}\label{eq_forms}
\sum_{l_1+\ldots+l_{r+1}=dm+\mu}F^\ell_{l_0\cdots l_{r+1}}(\lambda_{i,j})u_0^{l_0}\cdots u_{r+1}^{l_{r+1}}=\\
=\phi(u_0, \ldots, u_{r+1})\left(\sum_{i_1+\ldots+i_{r+1}=dm-\overline{m}+\mu}\alpha^\ell_{i_0\cdots i_{r+1}}u_0^{i_0}\cdots u_{r+1}^{i_{r+1}}\right)
\end{multline}
where the $\alpha^\ell_{i_0\cdots i_{r+1}}$'s are again indeterminates. Their number is 
\begin{equation*}\label{num_alpha}
\binom{d+k}{k}\binom{dm-\overline{m}+\mu+r+1}{r+1}.
\end{equation*}

Now to prove the thesis we need to show that, under condition \eqref{condizione-k-piano}, there exists an \emph{admissible solution} of the system of non-homogeneous equations obtained by equating the coefficients of the monomials of degree $dm+\mu$ in \eqref{eq_forms} for each $\ell=1, \ldots, \binom{d+k}{k}$. A solution of this system is called \emph{admissible} if it gives rise to a section. Clearly, a solution is admissible if and only if not all the $\lambda$'s are equal to $0$.

In the system there are 
\begin{equation*}\label{num_eq}
\binom{d+k}{k}\binom{dm+\mu+r+1}{r+1}
\end{equation*}
equations in the $\alpha$'s and $\lambda$'s. The total amount of these variables is  
\begin{equation*}\label{dim_Aff}
\begin{array}{c}
\displaystyle
[(k+1)(n-k)+1]\left[\binom{m+r+1}{r+1}-\binom{m-\overline{m}+r+1}{r+1}\right]\\
[5mm]
\displaystyle
+\binom{d+k}{k}\binom{dm-\overline{m}+\mu+r+1}{r+1}.
\end{array}
\end{equation*}
We claim that if the number of variables is greater than the number of equations, i.e., if the following inequality holds
\begin{equation}\label{cond_underd}
\begin{array}{c}
\displaystyle
[(k+1)(n-k)+1]\left[\binom{m+r+1}{r+1}-\binom{m-\overline{m}+r+1}{r+1}\right] \\
[5mm]
\displaystyle
+\binom{d+k}{k}\binom{dm-\overline{m}+\mu+r+1}{r+1} > \binom{d+k}{k}\binom{dm+\mu+r+1}{r+1}
\end{array}
\end{equation}
our system has admissible solutions and we do have sections as required. 

In general, given a system of non-homogeneous equations, it is not true that if it is \emph{underdeterminate} (i.e., the number of equations is lower than the number of the variables) then the set of solutions is non-empty. However we do know that, in the associated affine space with coordinates the $\lambda$'s and the $\alpha$'s, the origin,  where all $\lambda$'s and all $\alpha$'s vanish, is a  solution of the system, although it does not give rise to an admissible solution. In any event,  this implies that the set of solutions has a component $\mathfrak S$ of positive dimension which contains the origin. Moreover, $\mathfrak S$ cannot be contained in the subspace defined by the vanishing of all the $\lambda$'s. Indeed, if all the $\lambda$'s are equal to $0$, from \eqref{eq_forms} it follows that also the $\alpha$'s are $0$. This proves that if \eqref {cond_underd} holds, there are admissible solutions and therefore there are sections as desired. 

Finally we want to see under which conditions, for $m$ large enough, \eqref{cond_underd} holds. This can be written as
\begin{eqnarray*}\label{cond_underd 2}
[(k+1)(n-k)+1]\left[\binom{m+r+1}{r+1}-\binom{m-\overline{m}+r+1}{r+1}\right]+ \notag \\
\binom{d+k}{k}\left[\binom{dm-\overline{m}+\mu+r+1}{r+1} - \binom{dm+\mu+r+1}{r+1} \right]> 0
\end{eqnarray*}

The term on the left is a polynomial in $m$: the condition in order that it is positive for $m \gg 0$ is that the leading coefficient is positive. The coefficient of the monomial $m^{r+1}$ of maximal degree is equal to zero, so we have to look at the coefficient of $m^r$. This equals
\[
\frac{[(k+1)(n-k)+1]}{(r+1)!}(r+1)\overline{m}+ \frac{\binom{d+k}{k}d^r}{(r+1)!}[-(r+1)\overline{m}].
\]
After dividing for the positive term $\frac{\overline{m}}{r!}$, we obtain
\[
(k+1)(n-k)+1-\binom{d+k}{k}d^r
\]
and being this positive is equivalent to \eqref{condizione-k-piano}.
\end{proof}

\begin{remark}	\label{rem:ratdet} Note that the result of the Section Lemma is equivalent to say that if \eqref {condizione-k-piano} holds, and if $w$ is the generic point of $W$, that is defined over of the field of rational functions $\KK(W)$, then one can find a $k$--plane  $\Lambda$ in the generic hypersurface $X_w$ of the family, that is also defined over $\KK(W)$. In this case one says that $\Lambda$ is \emph{rationally determined} on $X_w$. 
\end{remark}

\section{Unirationality of families of hypersurfaces}\label{SEC_unir}

In this section we use the previous results to give a criterion for the unirationality of families of hypersurfaces. We need some preliminaries.

\subsection{} We recall the following: 

\begin{definition}\label{raz_con_irraz}
Let $X\subset \PP^n$ be an algebraic variety defined over $\KK$ and $\Lambda$ a $k$--plane contained in $X$. One says that $X$ is $\Lambda$--\emph{rational} (resp. $\Lambda$--\emph{unirational}) if $X$ is $\KK(\Lambda)$--rational (resp. $\KK(\Lambda)$--unirational), where $\KK(\Lambda)$ is the extension of $\KK$ obtained by adding to $\KK$ the Pl{\"u}cker coordinates of $\Lambda$.
\end{definition} 

Let $\mathscr{X} \rightarrow W$ be a flat family of subvarieties of $\PP^n$ with $W$ an irreducible variety. If $w\in W$ we denote, as usual, by $X_w\subset \PP^n$ the fibre of 
$\mathscr{X} \rightarrow W$ over $w$. We assume that there is a dense open subset $U$ of $W$ such that for all $w\in U$, $X_w$ is irreducible. So, up to shrinking $W$, we may assume that this happens for all $w\in W$. Let $F_k(\mathscr{X})\rightarrow W$ be the \emph{relative Fano scheme} of $k$--planes  of  $\mathscr{X} \rightarrow W$. For all $w\in W$, the fibre of $F_k(\mathscr{X})\rightarrow W$ is $F_k(X_w)$. 

The following criterion is due to Roth (see \cite{Roth}):

\begin{proposition}[Roth's Criterion]\label{criterion}
Let $\mathscr{X} \rightarrow W$ be a flat family of varieties with $W$ an irreducible, unirational variety. Suppose that $F_k(\mathscr{X})\rightarrow W$ is dominant, so that, up to shrinking $W$ we may assume it is flat. Suppose that there is a section $s: W\to F_k(\mathscr{X})$ of $F_k(\mathscr{X})\rightarrow W$ such that there is a dense open subset 
$U$ of $W$ such that for all $w\in U$ the variety $X_w$ is $s(w)$--unirational. Then $\mathscr{X}$ is unirational. 

In addition, if $W$ is rational and for all $w\in U$ the variety $X_w$ is $s(w)$--rational, then $\mathscr{X}$ is rational. \end{proposition}

\begin{proof} We may assume that $U=W$. Let $\phi:\mathbb{P}^{r} \dashrightarrow W$ be the dominant map which assures the unirationality of $W$ and by $\psi_w:\mathbb{P}^{r'}_{\KK(s(w))} \dashrightarrow X_w$ the dominant map which assures the unirationality of $X_w$, for $w \in W$.

Then we can construct the map
\[
\mathbb{P}^r \times \mathbb{P}^{r'} \dashrightarrow \mathscr{X}
\]
such that the pair $(t,t')$ is sent to $\psi_{\phi(t)}(t')$. This is a rational dominant map, and it is defined over $\KK$. 

It follows furthermore that if $\phi$ and $\psi_w$ are generically finite of degree $a$ and $b$ respectively, then this map is generically finite of degree $a \cdot b$. The second assertion follows.
\end{proof}

\subsection{} In the paper \cite {Pre}, A. Predonzan proved the following:

\begin{theorem}\label{thm:pred} Let $X\subset \PP^n$ be an irreducible hypersurface of degree $d\geq 2$ defined over $\KK$. Suppose that $X$ contains a $k$--plane $\Lambda$ 
with 
$$
k\geq k(d)
$$
where $k(d)$ is inductively defined as follows
$$
k(d)={{k(d-1)+d-1}\choose {d-1}}, \quad k(2)=0.
$$
Suppose that $X$ is smooth along $\Lambda$. Then $X$ is $\Lambda$--unirational. \end{theorem}

As a consequence of this result, Predonzan also proved in \cite {Pre} the following:

\begin{theorem}\label{thm:pred1} Let $X\subset \PP^n$ be an irreducible hypersurface of degree $d\geq 2$ defined over $\KK$, with a singular locus of dimension $t$. If
$$
n\geq \frac 1{k(d)+1}{{k(d)+d}\choose {d}}+k(d)+t+1
$$
then $X$ is unirational over an extension of $\KK$. 
\end{theorem}

This result has been rediscovered in \cite {HMP}, although with a worse lower bound for $n$. 
Our aim is to prove the following extension of Theorem \ref  {thm:pred1}:

\begin{theorem}\label{thm:pred2} Let $\mathscr{X}\to W$ be a family of hypersurfaces of degree $d\geq 2$ in $\PP^n$, with $W$ irreducible, unirational of dimension $r$. Assume that
if $w\in W$ is the generic point, then $X_w$ is irreducible with a singular locus of dimension $t$. If
\begin{equation}\label{condizione-k-piano-bis}
n>k(d)+\frac{1}{k(d)+1}\left[\binom{d+k(d)}{k(d)}d^r-1\right]+t+1.
\end{equation}
then $\mathscr{X}$ is unirational. 
\end{theorem}

\begin{proof} By the hypotheses, up to shrinking $W$ we may assume that for all $w\in W$ the hypersurface $X_w\subset \PP^n$ is irreducible with singular locus of dimension $t$. Again up to shrinking $W$, we may assume that there is a $(n-t-1)$--plane $P$ in $\PP^n$ such that for all $w\in W$, the intersection of $X_w$ with $P$ is smooth. In this way we get a new family $\mathscr{X}'\to W$ of hypersurfaces of degree $d$ in $\PP^{n-t-1}$ such that for all $w\in W$, $X'_w$ is the intersection of $X_w$ with $P$. 

Taking into account \eqref {condizione-k-piano-bis}, by the Section Lemma \ref {rat_sec_F},  up to shrinking $W$ we may assume there is a section $s$ of $F_{k(d)}(\mathscr{X}')\to W$. Note that for all $w\in W$, $X'_w$ is smooth, and therefore $X_w$ is smooth along $s(w)$. Then, by Theorem \ref {thm:pred}, for all $w\in W$,  $X_w$ is  $s(w)$--unirational. Thus, by applying Roth's Criterion  \ref {criterion}, the assertion follows. \end{proof} 

We notice that if $d=2$ then Theorem \ref {thm:pred2} is basically the main result of \cite {conf:unisec}.

\end{document}